\documentclass[11pt,twoside,a4paper]{amsart}

\usepackage{graphicx}
\usepackage{enumerate}
\usepackage{wrapfig}
\usepackage{pgfplots}
\usepackage{amsfonts}
\usepackage{amsmath}
\usepackage{amsthm}
\usepackage[utf8]{inputenc}
\usepackage{xy} \xyoption{all}
\usepackage{setspace}
\usepackage[]{mathtools}
\usepackage{amsmath}
\usepackage{amsthm}
\usepackage{amssymb}
\usepackage{mathabx}
\usepackage{multicol}
\usepackage{calc}
\usepackage{lipsum}
\usepackage{colortbl}
\usepackage{hyperref}
\usepackage{geometry}
\usepackage{tikz-cd}

\theoremstyle{plain}
\newtheorem{theorem}{Theorem}

\newtheorem{corollary}[theorem]{Corollary}

\newtheorem{proposition}[theorem]{Proposition}
\newtheorem{lemma}[theorem]{Lemma}

\theoremstyle{definition}
\newtheorem{definition}[theorem]{Definition}
\newtheorem{example}[theorem]{Example}

\newtheorem{notation}[theorem]{Notation}

\newtheorem{assumption}[theorem]{Assumption}

\theoremstyle{remark}
\newtheorem{remark}[theorem]{Remark}

\numberwithin{theorem}{section}

\usetikzlibrary{arrows,automata,positioning}
\usetikzlibrary{decorations.markings}
\usetikzlibrary{shapes,arrows}
\usetikzlibrary{positioning}
\usetikzlibrary{arrows}
\usetikzlibrary{calc,shapes}
\usetikzlibrary{matrix, fit}
\usetikzlibrary{backgrounds}

\tikzstyle{vertex}=[circle, draw, inner sep=0pt, minimum size=6pt]

\tikzset{style green/.style={
		set fill color=green!50!lime!60,
		set border color=white,
	},
	style cyan/.style={
		set fill color=cyan!90!blue!60,
		set border color=white,
	},
	style orange/.style={
		set fill color=orange!30,
		set border color=white,
	},
	style blue/.style={
	set fill color=blue!20,
	set border color=blue!20,
	},
	hor/.style={
		above left offset={-0.100,0.31},
		below right offset={0.15,-0.125},
		#1
	},
	ver/.style={
		
	above left offset={-0.1,0.3},
		below right offset={0.15,-0.15},
		#1
	}
}

\title{Leavitt path algebras of weighted Cayley graphs $C_n(S,w)$}
\author{Mohan.R}

\address{Statistics and Mathematics Unit, Indian Statistical Institute Bangalore, India}
\email{rmohan689@gmail.com}

\begin{document}

\begin{abstract}
For a postive integer $n$ and a subset $S$ of $\mathbb{Z}_n$, let $\left\langle S\right\rangle =\mathbb{Z}_n$, and $w:S\rightarrow\mathbb{N}$ be a function. The weighted Cayley graph of the cyclic group $\mathbb{Z}_n$ with respect to $S$ and $w$ is denoted by $C_n(S,w)$. We give an explicit description of the Grothendieck group of the Leavitt path algebras of $C_n(S,w)$. We also give description of Leavitt path algebras of $C_n(S,w)$ in some special cases. 
\end{abstract}
\maketitle
\section{Introduction}
For a finite group $G$ and a subset $S\subseteq G$, let the associated Cayley graph be denoted by $\mathsf{Cay}(G,S)$. When the given group is $\mathbb{Z}_n$ we write $C_n(S):=\mathsf{Cay}(\mathbb{Z}_n,S)$. Leavitt path algebras of Cayley graphs of the finite cyclic group $\mathbb{Z}_n$ with respect to the subset $S=\{1,n-1\}$ were initially studied in \cite{MR3307385}. It was shown that there are exactly four isomorphism classes represented by the collection $\{L(C_n(1,n-1))\mid n\in\mathbb{N}\}.$ 

In subsequent work, \cite{MR3456905} contains the computation of the important integers $|K_0(L(C_n(1,j)))|$ and $\det(I_n-A_{C_n(1,j)}^t)$, where $A_{(-)}$ denotes the adjacency matrix of a directed graph, and $K_0(-)$ denotes the Grothendieck group of a ring. Also in \cite{MR3456905} the collections of $K$-algebras were described upto isomorphism:
$$\{L(C_n(1,j))\mid n\in \mathbb{N}\}\quad \text{for}\quad j=0,1,2.$$
The descriptions of all these algebras follow from an application of the powerful tool known as the (Restricted) Algebraic Kirchberg-Philips Theorem.

In \cite{MR3854343} the study was extended and a method to compute the Grothendieck group of the Leavitt path algebra $L(C_n(1,j))$ to the case where $0\leq j\leq n-1$ and $n\geq3$ was derived.  Specifically it was shown how to reduce the computation of the Smith Normal Form of the $n\times n$ matrix $I_n-A_{C_n(1,j)}^t$ to that of calculating the Smith Normal Form of a $j\times j$ matrix $(M_j^n)^t-I_j$. Further a description of $K_0(L(C_n(1,j))$ was also given. 

In this paper we generalize the work done in \cite{MR3854343} to study
$L(C_n(S,w))$, where $S$ is any nonempty generating subset of $\mathbb{Z}_n$, $w:S\rightarrow\mathbb{N}$ is a map and $C_n(S,w)$ is the weighted Cayley graph. In section 2 we recall the background information required. In Section 3 we present a method to compute the Grothendieck group. Specifically we find the conditions to determine the sign of $\det(I_n-A_{C_n(S,w)}^t)$ and also the cardinality of $K_0(L(C_n(S,w)))$. Also we find a method to reduce the computation of the Smith Normal form of the $n\times n$ matrix $I_n-A_{C_n(S,w)}^t$ to that of calculating the Smith Normal form of a square matrix of smaller size if $0\notin S$ (Theorem \ref{Coker}). In Section 4 we use the method developed in Section 3 to study the following simple cases when $\left\langle S\right\rangle =\mathbb{Z}_n$:
\begin{enumerate}
	\item[Case 1]: $|S|=1$
	\item[Case 2]: $|S|=2$
	\item[Case 3]: $|S|=n$ 
\end{enumerate}
Moreover we recover the results studied in \cite{MR3307385},\cite{MR3456905}, and \cite{MR3854343} as special cases and get some new results. Among these new results, in particular, we show that $L(K_n)\cong L(1,n)$ where $K_n$ is the unweighted complete $n$-graph (See \ref{def:K_n} for definition) and $L(1,n)$ is the Leavitt algebra.  We also show that the main result of \cite{MR3307385}  holds true if $C_n(1,n-1)$ is replaced by $D_n$ for every $n\in\mathbb{N}$, where $D_n$ denotes the Cayley graph of Dihedral group with respect to usual generating set. 
\section{Preliminaries}

\begin{notation}
    Throughout by $K$ we mean a fixed field. $\mathbb{N}$ denotes the set of natural numbers, $\mathbb{Z}^+$ denotes the set of non-negative integers, $\mathbb{Z}$ denotes the set of integers and $\mathbb{Q}$ denotes the set of rationals. $|S|$ is the cardinality of the set $S$.
\end{notation}
\subsection{Leavitt path algebras and the Algebraic KP theorem}\hfill\\

A \textit{graph} $E=(E^0,E^1,s,r)$ consists of two sets $E^0,E^1$ and functions $s,r:E^1\rightarrow E^0$. The elements of $E^0$ are called vertices and the elements of $E^1$ are called edges. If $e$ is an edge, then $s(e)$ is called its source and $r(e)$ its range. $E$ is called finite if $E^0$ and $E^1$ are finite sets. 

A vertex $v$ is called a \textit{source} (resp. \textit{sink}) if $r^{-1}(v)=\emptyset$ (resp. if $s^{-1}(v)=\emptyset$). A graph is called \textit{sink-free} (resp. source-free) if it has no sinks (resp. no sources). A non-sink $v\in E^0$ is called \textit{regular} is $s^{-1}(v)$ is finite. ($v$ emits only finitely many edges). 

A \textit{path} $\mu$ in a graph $E$ is either a vertex in $E$ or a finite sequence $e_1e_2\dots e_n$ of edges in $E$ such that $r(e_i)=s(e_{i+1})$ for $1\leq i\leq n-1$. A path $\mu=e_1e_2\dots e_n$ for which $n\leq 1$ and $s(e_1)=r(e_n)=v$ is called a closed path based at $v$. A closed path $\mu=e_1e_2\dots e_n$ based at $v$ for which $s(e_i)\neq s(e_j)$ for any $i\neq j$ is called a \textit{cycle} (based at $v$). A graph which contains no cycles is called \textit{acyclic}.

Let $E=(E^0,E^1,s,r)$ be graph. The \textit{adjacency matrix} $A_E$ of $E$ is the $|E^0|\times|E^0|$ matrix whose entries are given by
$$A_E(v,w)=|\{e\in E^1\mid s(e)=v,~ r(e)=w\}|~\text{for any}~v,w\in E^0$$

\begin{definition}\label{def: cofinal and condition L}
	Let $E$ be a graph and $H\subseteq E^0$.
	\begin{enumerate}
		\item $H$ is \textit{hereditary} if whenever $v\in H$ and $w\in E^0$ for which there exists a path $\mu$ such that $s(\mu)=v$ and $r(\mu)=w$, then $w\in H$.
		\item $H$ is \textit{saturated} if whenever $v\in E^0$ is regular such that $\{r(e)\mid e\in E^1, s(e)=v\}\subseteq H$, then $v\in H$.
		\item $E$ satisfies \textit{condition (L)} if every cycle in $E$ has an exit.
	\end{enumerate}
\end{definition}
\begin{definition}\label{def:Leavitt path algebra}
	\label{LPA}
Let $K$ be a field and let $E=(E^0,E^1,r,s)$ be a graph. The \textit{Leavitt path algebra} of $E$ is the $K$-algebra presented by generators $E^0\sqcup E^1\sqcup \overline{E^1}$ where $\overline{E^1}:=\{e^\ast\mid e\in E^1\}$  and the following relations.
\begin{enumerate}
    \item[(V)] $\forall ~v,w\in E^0$, $vw=\delta_{vw}v$
    \item[(E)] $\forall~e\in E^1$, $s(e)e=e=er(e)$ and $r(e)e^\ast=e^ast=e^\ast s(e)$
    \item[(CK1)] $\forall~e,f\in E^1$, $e^\ast f=\delta_{ef}r(e)$
    \item[(CK2)] $\forall~v\in E^0$ if $0<|s^{-1}(v)|<\infty$ then $v=\sum\limits_{e\in s^{-1}(v)}ee^\ast$
\end{enumerate}
where $\delta_{ij}$ is Kronecker delta. We denote the Leavitt path algebra by $L(E)$ if underlying field $K$ is fixed.
\end{definition}

\begin{remark}\label{remark: LPA is unital}
	\label{unital}
	It is easy to see that $L(E)$ is unital if and only if $|E^0|$ is finite, in which case $\sum\limits_{v\in E^0}v$ acts as the unity. In this paper we focus only on finite graphs. 
\end{remark}

One of the motivations to define Leavitt path algebra is to study a natural abelian monoid associated to a given graph $E$ called the graph monoid $M_E$ which we define below.

\begin{definition}\label{graph monoid}
Let $E$ be a finite graph with vertex set $E^0$ and adjacency matrix $A_E=(a_{vw})$. The \textit{graph monoid} $M_E$ of $E$ is the abelian monoid presented by generating set $E^0$ and following relations:
$$v = \sum_{w\in E^0}a_{vw}w$$
\end{definition}

Recall that for a unital $K$-algebra $R$, the \textit{V-monoid of $R$}, denoted by $\mathcal{V}(R)$,  is the set of isomorphism classes of finitely generated projective left $R$-modules. We denote the elements of $\mathcal{V}(R)$ using brackets, for example, $[R]\in \mathcal{V}(R)$ represents the isomorphism class of the left regular module $\prescript{}{R}{R}$. Then $\mathcal{V}(R)$ is an abelian monoid, with operation $\oplus$, and zero element $[0]$, where $0$ is the zero $R$-module. Also, the moniod $(\mathcal{V}(R),\oplus)$ is conical; that is the sum of any two nonzero elements of $\mathcal{V}(R)$ is nonzero, or rephrased, $\mathcal{V}(R)^\ast=\mathcal{V}(R)-[0]$ is a semigroup under $\oplus$. The group completion of $\mathcal{V}(R)$ is denoted by $K_0(R)$ and called the \textit{Grothendieck group} of $R$.

\begin{theorem}\cite[Theorem 3.5]{MR2310414}\label{theorem: Vmoniod is graph monoid}
	As monoids, $\mathcal{V}(L(E))\cong M_E$ and $[L(E)] \leftrightarrow \sum\limits_{v\in E^0} [v] $ under this isomorphism.
\end{theorem}

\begin{definition}\label{def: purely infinite simple}
	A unital $K$-algebra $A$ is called \textit{purely infinite simple} in case $A$ is not a division ring, and $A$ has the property that for every nonzero element $x$ of $A$ there exists $b,c \in A$ for which $bxc=1_A$.
\end{definition}

The finite graphs $E$ for which the Leavitt path algebra $L(E)$ is purely infinite simple have been explicitly  described in \cite{MR2265539}: 

\begin{theorem}\label{theorem: LPA is pis iff}
$L(E)$ is purely infinite simple if and only if $E$ is sink-free, satisfies Condition ($L$), and only hereditary and saturated subsets of $E^0$ are $\emptyset$ and $E^0$.
\end{theorem}

In other words, the graph $E$ satisfies the following properties: every vertex in $E$ connects to every cycle of $E$; every cycle in $E$ has an exit; and $E$ contains at least one cycle.

It is shown in \cite[Corollary 2.2]{MR1918211}, that if $A$ is a unital purely infinite simple $K$-algebra, then the semigroup $(\mathcal{V}(A)^\ast,\oplus)$ is in fact a group, and moreover, that $\mathcal{V}(A)^\ast \cong K_0(A)$, the Grothendieck group of $A$. For unital Leavitt path algebras, the converse is true as well: if $\mathcal{V}(L(E))^\ast$ is a group, then $L(E)$ is purely infinite simple. (This converse is not true for general $K$-algebras.)  

\begin{theorem}\label{theorem: pis implies K0 isomorphic to M*}
If $L(E)$ is unital purely infinite simple then 
$$K_0(L(E))\cong \mathcal{V}(L(E))^\ast \cong M_E^\ast.$$

\end{theorem}

The following important theorem will be used to yield a number of key results in the following sections.

\begin{theorem}[\textbf{(Restricted) Algebraic KP Theorem}]\cite[Corollary 2.7]{MR2785945}\label{theorem: KP}\\
	Suppose $E$ and $F$ are finite graphs for which the Leavitt path algebras $L(E)$ and $L(F)$ are purely infinite simple. Suppose that there is an isomorphism $\varphi : K_0(L(E)) \rightarrow K_0(L(F))$ for which $\varphi([L(E)])=[L(F)]$, and suppose also that the two integers $\emph{det}(I_{|E^0|}-A_E^t)$ and $\emph{det}(I_{|F^0|}-A_F^t)$ have the same sign. Then $L(E)\cong L(F)$ as $K$-algebras.
\end{theorem}

\begin{example}[Leavitt algebras]
For any integer $m\geq 2$, $L(1,m)$ is the free associative $K$-algebra in $2m$ generators $x_1, x_2, \dots, x_m,$ $y_1, y_2, \dots, y_m$, subject to the relations $$y_ix_j=\delta_{i,j}1_K\quad \text{and}\quad  \sum\limits_{i=1}^mx_iy_i=1_K.$$
These algebras were first defined and investigated in \cite{MR0132764} in the context of finding counter-examples for the invariant basis number problem, and formed the motivating examples for the more general notion of Leavitt path algebra. It is easy to see that for $m>2$, if $R_m$ is the graph having one vertex and $m$ loops, then $L(R_m)\cong L(1,m)$. From Theorem \ref{theorem: LPA is pis iff} it follows that $L(R_m)$ is unital purely infinite simple and hence $K_0(L(R_m))\cong M_{R_m}^\ast$ is the cyclic group $\mathbb{Z}_{m-1}$, where the regular module $[L(R_m)]$ in $K_0(L(R_m))$ corresponds to 1 in $\mathbb{Z}_{m-1}$.    
\end{example}

Unital purely infinite simple Leavitt path algebras $L(E)$ whose corresponding $K_0$ groups are cyclic and for which det($I_{|E^0|}-A_E^t)\leq 0$ are relatively well-understood, and arise as matrix rings over the Leavitt algebras $L(1,m)$, as follows. Let $d\geq 2$, and consider the graph $R_m^d$ having two vertices $v_1, v_2; d-1$ edges from $v_1$ to $v_2$; and $m$ loops at $v_2$.

\begin{figure}[h]
	\center
	\begin{tikzpicture}
	[->,>=stealth',shorten >=1pt,thick]
	\draw [fill=black] (0,0) circle [radius=0.1];
	\draw [fill=black] (2,0) circle [radius=0.1];

	\node at (0,-0.5) {$v_1$};
	\node at (2,-0.5) {$v_2$};
	
	\node at (1,0.5) {$(d-1)$};
	\node at (3.6,0) {$(m)$};	
	\node at (-1,0) {$R_m^d=$};

	\draw[->] (0.2,0) -- (1.8,0);
	
	\path[->] (2.3,0) edge [out=40, in=320, looseness=4, bend right, distance=2cm, ->] node {} (2.3,0);

	\end{tikzpicture}
\end{figure}
It is shown in \cite{MR2437640} that $L(R_m^d)$ is isomorphic to the matrix algebra $M_d(L(1,m))$. By standard Morita equivalence theory we have that $K_0(M_d(L(1,m)))\cong K_0(L(1,m))$. Moreover, the element $[M_d(L(1,m))]$ of $K_0(M_d(L(1,m)))$ corresponds to the element $d$ in $\mathbb{Z}_{m-1}$. In particular, the element $[M_{m-1}(L(1,m))]$ of $K_0(M_{m-1}(L(1,m)))$ corresponds to $m-1 \equiv 0$ in $\mathbb{Z}_{m-1}$. Finally, an easy computation yields that det$(I_2-A_{R_m^d}^t)=-(m-1)\leq 0$ for all $m,d$. Therefore, by invoking the Algebraic KP Theorem, the previous discussion immediately yields the following.

\begin{proposition}\label{prop: Leavitt algebra}
	Suppose that $E$ is a graph for which $L(E)$ is unital purely infinite simple. Suppose that $M_E^*$ is isomorphic to the cyclic group $\mathbb{Z}_{m-1}$, via an isomorphism which takes the element $\sum_{v\in E^0}[v]$ of $M_E^*$ to the element $d$ of $\mathbb{Z}_{m-1}$. Finally, suppose that det$(I_{|E^0|}-A_E^t)\leq0$. Then $L(E)\cong M_d(L(1,m))$.
\end{proposition}

\subsubsection{\textbf{Computation of Grothendieck group}}\hfill\\

Let $E$ be a finite directed graph for which $|E^0|=n$. We view $I_n-A_E^t$ both as a matrix, and as a linear transformation $I_n-A_E^t:\mathbb{Z}^n\rightarrow\mathbb{Z}^n$, via left multiplication (viewing elements of $\mathbb{Z}^n$ as column vectors). As discussed in \cite[Section 3]{MR2437640}, we have

\begin{proposition}\label{prop: K0 is cokernal}
If $L(E)$ is purely infinite simple, then
$$M_E^\ast \cong K_0(L(E))\cong \mathbb{Z}^n/\emph{Im}(I_n-A_E^t)=\emph{Coker}(I_n-A_E^t).$$
Under this isomorphism $[v_i]\mapsto \vec{b_i}+\emph{Im}(I_n-A_E^t)$, where $\vec{b_i}$ is the element of $\mathbb{Z}^n$ which is 1 in the $i^{th}$ coordinate and 0 elsewhere.
\end{proposition}

Let $M\in M_n(\mathbb{Z})$ and view $M$ as a linear transformation $M:\mathbb{Z}^n\rightarrow\mathbb{Z}^n$ via left multiplication on columns. The cokernel of $M$ is a finitely generated abelian group, having at most $n$ summands; as such, by the invariant factors version of the Fundamental Theorem of Finitely Generated Abelian Groups, we have 
$$\text{Coker}(M)\cong \mathbb{Z}_{s_l}\oplus\mathbb{Z}_{s_{l+1}}\oplus\dots\oplus\mathbb{Z}_{s_{n}},$$ 
for some $1\leq l\leq n$, where either $n=l$ and $s_n=1$ (i.e., Coker($M$) is trivial group), or there are (necessarily unique) nonnegative integers $s_l, s_{l+1},\dots,s_n$, for which the nonzero values $s_l, s_{l+1},\dots,s_r$ satisfy $s_j\geq2$ for $1\leq j\leq r$ and $s_i|s_{i+1}$ for $l\leq i\leq r-1$, and $s_{r+1}=\dots=s_n=0$. Coker($M$) is a finite group if and only if $r=n$. In case $l>1$, we define $s_1=s_2=\dots=s_{l-1}=1$. Clearly then we have $$\text{Coker}(M)\cong\mathbb{Z}_{s_1}\oplus\mathbb{Z}_{s_2}\oplus\dots\oplus\mathbb{Z}_{s_l}\oplus\dots\oplus\mathbb{Z}_{s_n},$$ since any additional direct summands are isomorphic to the trivial group $\mathbb{Z}_1$.

We note that if $P,Q$ are invertible in $M_n(\mathbb{Z})$ (hence their determinant is $\pm1$), then Coker$(M)\cong\text{Coker}(PMQ)$. In other words, if $N\in M_n(\mathbb{Z})$ is a matrix which is constructed by performing any sequence of $\mathbb{Z}$-elementary row (or column) operations starting with $M$, then $\text{Coker}(M)\cong\text{Coker}(N)$ as abelian groups. 

\begin{definition}\label{def: SNF}
	Let $M\in M_n(\mathbb{Z})$, and suppose $\text{Coker}(M)\cong\mathbb{Z}_{s_1}\oplus\mathbb{Z}_{s_2}\oplus\dots\oplus\mathbb{Z}_{s_n}$ as described above. The \textit{Smith Normal Form} of $M$ ((SNF($M$) to be short)) is the $n\times n$ diagonal matrix $\text{diag}(s_1, s_2,\dots,s_r,0,\dots,0)$.
\end{definition} For any matrix $M\in M_n(\mathbb{Z})$, the Smith Normal Form of $M$ exists and is unique. If $D\in M_n(\mathbb{Z})$ is a diagonal matrix with entries $d_1,d_2,\dots,d_n$ then clearly $\text{Coker}(D)\cong\mathbb{Z}_{d_1}\oplus\mathbb{Z}_{d_2}\oplus\dots\oplus\mathbb{Z}_{d_n}$. We also note the following.

\begin{proposition}\label{prop: SNF gives Coker}
	Let $M\in M_n(\mathbb{Z})$, and let $S$ denote the Smith Normal Form of $M$. Suppose the diagonal entries of $S$ are $s_1, s_2, \dots, s_n$. Then $$\emph{Coker}(M)\cong\mathbb{Z}_{s_1}\oplus\mathbb{Z}_{s_2}\oplus\dots\oplus\mathbb{Z}_{s_n}.$$ In particular, if there are no zero entries in the Smith Normal Form of $M$, then $|\emph{Coker}(M)|=s_1s_2\dots s_n=|\emph{det}(S)|=|\emph{det}(M)|$.
\end{proposition}

Proposition \ref{prop: SNF gives Coker} yields the following
\begin{proposition}\label{prop: K_0 of LPA}
	Let $E$ be a finite graph with $|E^0|=n$ and adjacency matrix $A_E$. Suppose that $L(E)$ is purely infinite simple. Let $S$ be the Smith Normal Form of the matrix $I_n-A_E^t$, with diagonal entries $s_1, s_2, \dots, s_n$. Then $$K_0(L(E))\cong\mathbb{Z}_{s_1}\oplus\mathbb{Z}_{s_2}\oplus\dots\oplus\mathbb{Z}_{s_n}.$$ Moreover, if $K_0(L(E))$ is finite, then an analysis of the Smith Normal Form of the matrix $I_n-A_E^t$ yields $$\left| K_0(L(E))\right| =\left| \emph{det}(I_n-A_E^t)\right|,$$ Conversely, $K_0(L(E))$ is infinite if and only if $\emph{det}(I_n-A_E^t)=0$ and in this case $\emph{rank}(K_0(L(E))=\emph{nullity}(I_n-A_E^t)$.
\end{proposition}

We record the following theorem which will be used in computations of Smith Normal Forms in later sections
\begin{theorem}[\textbf{Determinant Divisors Theorem}]\cite[Theorem II.9]{MR0340283}\label{DDT}\\
Let $M\in M_n(\mathbb{Z})$. Define $\alpha_0:=1$, and for each $1\leq i\leq n$, define the $i^{th}$ determinant divisor of $M$ to be the integer 
	\begin{center}
		$\alpha_i:=$ the greatest common divisor of the set of all $i\times i$ minors of $M$.
	\end{center}
	Let $s_1, s_2, \dots, s_n$ denote the diagonal entries of the Smith Normal Form of $M$, and assume that each $s_i$ is nonzero. Then $$s_i=\frac{\alpha_i}{\alpha_{i-1}}$$ for each $1\leq i\leq n$.
\end{theorem}

\subsection{Weighted Cayley graphs and circulant matrices}\label{cayley prelims}\hfill\\

Let $E$ be a graph and $w:E^1\rightarrow\mathbb{N}$ be a function, the \textit{weighted graph} of $E$ associated to $w$ is a new graph $E_w=(E^0,E_w^1,s_w,r_w)$ where $E_w^1:=\{e_1\dots e_{w(e)}\mid e\in E^1\}, s_w(e_i)=s(e)$ and $r_w(e_i)=r(e)$.

Recall that given a group $G$, and a subset $S\subseteq G$, the \textit{associated Cayley graph} $\mathsf{Cay}(G,S)$ is the directed graph $E(G,S)$ with vertex set $\left\lbrace v_g\mid g\in G\right\rbrace $, and in which there is an edge $e(g,h)$ from $v_g$ to $v_h$ in case there exists (a necessarily unique) $s\in S$ with $h=gs$ in $G$. Thus, in $\mathsf{Cay}(G,S)$, at every vertex $v_g$, the number of edges emitted is $|S|$. The identity of $G$ is in $S$ if and only  $\mathsf{Cay}(G,S)$ contains a loop at every vertex.

\begin{definition}\label{def: w-Cayley graph} 
Let $G$ be a group, $S\subseteq G$ and $w:S\rightarrow \mathbb{N}$ be a map. Then $w$ induces a map (also denoted by $w$) from the set of edges of $\mathsf{Cay}(G,S)$ to $\mathbb{N}$ by $e(g,h) \mapsto w(s)$ whenever $h=gs$. The weighted graph of $\mathsf{Cay}(G,S)$ associated to the map $w$ is called the \textit{weighted Cayley graph} (or \textit{w-Cayley graph}) and denoted by $\mathsf{Cay}(G,S,w)$
\end{definition}

In particular $\mathsf{Cay}(G,S)$ is a special case of $\mathsf{Cay}(G,S,w)$ when $w$ is the constant map $w(e)=1$ for every edge $e$. In this case we say $\mathsf{Cay}(G,S)$ is unweighted. 

\begin{remark}\label{remark: <S>=G means Cay is connected}
$\mathsf{Cay}(G,S,w)$ is strongly connected if and only if $\langle S\rangle=G$. In particular, $\mathsf{Cay}(\langle S\rangle,S,w)$ is a connected component of $\mathsf{Cay}(G,S,w)$, where $\langle S\rangle$ is the subgroup generated by $S$.
\end{remark}
\begin{notation}
	For a positive integer $n$, let $G=\mathbb{Z}_n$, and $S$ be any non-empty subset of $G$. We denote the w-Cayley graph $\mathsf{Cay}(G,S,w)$ simply by $C_n(S,w)$.

In other words, if $S=\left\lbrace s_1,s_2,\dots s_k\right\rbrace $ then the w-Cayley graph $C_n(S,w)$ is the directed graph with vertex set $\left\lbrace v_0,v_1,v_2,\dots,v_{n-1}\right\rbrace$ and edge set $\{e_l(i,s_j)\mid 0\leq i\leq n-1, 1\leq j\leq k, 1\leq l \leq w(s_j)\}$	for which $s(e_l(i,s_j))=v_i$, and $r(e_l(i,s_j))=v_{i+s_j}$ where indices are interpreted modulo $n$. Therefore $C_n(S,w)$ is a finite graph. 
\end{notation}

\begin{definition}\label{def: circulant matrix}
	For a positive integer $n$, let $\textbf{c}=\left( c_0,c_1,\dots,c_{n-1}\right) \in\mathbb{Q}^n$. Consider the shift operator $T:\mathbb{Q}^n\rightarrow\mathbb{Q}^n$, defined by $T(c_0,c_1,\dots,c_{n-1})=(c_{n-1},c_0,\dots,c_{n-2}).$ The \textit{circulant matrix} $circ(\textit{c})$, associated with $\textbf{c}$ is the  $n\times n$ matrix $C$ whose $k$th row is $T^{k-1}(\textbf{c})$, for $k=1,2,\dots.n$. Thus $C$ is of the form

\[C=\begin{pmatrix}
c_0 & c_1 & c_2 & \dots & c_{n-1} \\
c_{n-1} & c_0 & c_1 & \dots & c_{n-2}\\
\vdots & \vdots & \vdots & \dots & \vdots \\
c_1 & c_2 & c_3 & \dots & c_0 \\
\end{pmatrix}\]
\end{definition}

In other words, a circulant matrix is obtained by taking an arbitrary first row, and shifting it cyclically one position to the right in order to obtain successive rows. The $(i,j)$ element of $C$ is $c_{j-i}$, where subscripts are taken modulo $n$. 

Note that $A_{C_n(S,w)}$ is the $n\times n$ matrix with the $(i,j)^{th}$ entry is $w(s)$ if $i+s=j$ modulo $n$, for some $s\in S$, otherwise $0$. Hence $A_{C_n(S,w)}$ is a circulant matrix with non-negative integer entries. In the case of unweighted Cayley graph $C_n(S)$, the adjacency matrix is binary circulant matrix. 

\begin{definition}\label{def: representer polynomial}
	For $\textbf{c}\in\mathbb{Q}^n$, let $C=circ(\textbf{c})$. The \textit{representer polynomial} of $C$ is defined to be the polynomial $P_C(x)=c_0+c_1x+\dots+c_{n-1}x^{n-1}\in\mathbb{Q}[x]$.
\end{definition}
\begin{lemma}\label{lemma: eigenvalue of circulant matrix}
	Let $C=circ(\textbf{c})$ be a circulant matrix. Then the eigenvalues of $C$ equal $P_C(\zeta_n^k)=c_0+c_1\zeta_n^k+\dots+c_{n-1}\zeta_n^{k(n-1)}$ for $k=0,1,\dots,n-1$, where $\zeta_n=e^{\frac{2\pi i}{n}}$, the primitive $n^{th}$ root of unity. Further $$\det(C)=\prod\limits_{l=0}^{n-1}\left(\sum\limits_{k=0}^{n-1}c_j\zeta_n^{lk}\right).$$	
\end{lemma}
For a proof of Lemma \ref{lemma: eigenvalue of circulant matrix} we refer the reader to \cite[Theorem 6]{MR2931628}.

Note that the $n^\text{th}$ cyclotomic polynomial, denoted by $$\Phi_n(x)=\prod\limits_{\substack{1\le a < n\\ \gcd(a,n)=1}}(x-\zeta_n^a),$$ is an element of $\mathbb{Z}[x]$. Also, $x^n-1=\prod\limits_{d|n}\Phi_d(x)$. Since $\Phi_n(x)$ is the minimal polynomial of $\zeta_n$, $f(\zeta_n)=0$ for some $f(x)\in\mathbb{Z}[x]$ implies $\Phi_n(x)$ divides $f(x)$. By applying Lemma \ref{lemma: eigenvalue of circulant matrix} we get 
\begin{lemma}\label{lemma: circulant matrix is singular}
	Let $C=circ(\textbf{c})$. Then the following are equivalent.
	\begin{enumerate}
		\item[(a)] $C$ is singular 
		\item[(b)] $P_C(\zeta_n^k)=0$ for some $k\in\mathbb{Z}$.
		\item[(c)] The polynomials $P_C(x)$ and $x^n-1$ are not relatively prime.
	\end{enumerate}
\end{lemma}

\section{Leavitt path algebras of $C_n(S,w)$}

\begin{theorem}\label{theorem: wCayley is pis iff}
Let $G$ be a finite group, $S$ its generating set and $w:S\rightarrow\mathbb{N}$ a weight function. Let $W=\sum\limits_{s\in S}w(s)$. Then the following are equivalent.
\begin{enumerate}
    \item $L(\mathsf{Cay}(G,S,w)$ is purely infinite simple
    \item $W\geq 2$
    \item $L(\mathsf{Cay}(G,S,w)$ does not have Invariant Basis Number
\end{enumerate}
\end{theorem}
\begin{proof}

$(1)\Rightarrow(2)\Leftarrow(3)$. Let $|G|=n$. If $W=1$, then $|S|=1$. Setting $S=\{g\}$ we have $G$ is cyclic group generated by $g$. Hence $\mathsf{Cay}(G,S,w)$ is the graph $C_n$ which is cycle of length $n$ , which does not satisfy condition $L$, which contradicts $(1)$. By \cite[Theorem 3.8 and 3.10]{MR2403625} $L(C_n)\cong M_n(K[x,x^{-1}])$ which has Invariant Basis Number.

$(2)\Rightarrow(1)$. Let $W\geq 2$. In $\mathsf{Cay}(G,S,w)$, the number of edges emitted at each vertex $v_g$ is $W$. So there is at least two edges emitted from each vertex. This also implies condition $(L)$. Since $\langle S\rangle=G$, $\mathsf{Cay}(G,S,w)$ is strongly connected. Hence for any vertex $v_g$ there is a non-trivial path connecting $v_g$ to $v_1$ and vice versa. Therefore $\mathsf{Cay}(G,S,w)$ contains a cycle and there is no non-trivial hereditary subset of vertices.

$(2)\Rightarrow(3)$. For a finite graph $E$, $L(E)$ has Invariant Basis Number if and only if for each pair of positive integers $m$ and $n$, $$m\sum\limits_{v\in E^0}[v]=n\sum\limits_{v\in E^0}[v]~\text{in}~M_E\Rightarrow m=n.$$
In $M_{\mathsf{Cay}(G,S,w)}$, for each $v_g$ we have $[v_g]=\sum\limits_{s\in S}w(s)[v_{gs}]$ and hence,
$$\sum\limits_{g\in G}[v_g]=\sum\limits_{g\in G}\sum\limits_{s\in S}w(s)[v_{gs}]=\sum\limits_{s\in S}\sum\limits_{g\in G}w(s)[v_{gs}]=\sum\limits_{s\in S}w(s)\sum\limits_{g\in G}[v_{gs}]=W\sum\limits_{g\in G}[v_{gs}].$$
Since $G$ is a finite group we have $G=\{gs\mid g\in G\}$ and hence $\sum\limits_{g\in G}[v_{gs}]=\sum\limits_{g\in G}[v_g]$. Hence we have $\sum\limits_{g\in G}[v_g]=W\sum\limits_{g\in G}[v_g]$. If $W\geq 2$ then $L(\mathsf{Cay}(G,S,w))$ does not have Invariant Basis Number.
\end{proof}
\begin{corollary}\cite[Proposition 4.1, Theorem 4.2]{MR3955043}
Let $G$ be a finite group, $S$ its generating set. Then the following are equivalent.
\begin{enumerate}
    \item $L(\mathsf{Cay}(G,S)$ is purely infinite simple
    \item $L(\mathsf{Cay}(G,S)$ does not have Invariant Basis Number
    \item $|S|\geq 2$
\end{enumerate} 
\end{corollary}
\begin{proof}
    In this case $W=|S|$.
\end{proof}

From now on we work with the following assumption.
\begin{assumption}
Let $n\in\mathbb{N}$, $S=\{s_1,s_2,\dots s_k\}\subseteq\mathbb{Z}_n$. $s_1<s_2<\dots<s_k$. Further set $W:=\sum\limits_{s_j\in S}w(s_j)$.
\end{assumption}

\begin{theorem}\label{theorem: order of identity in K_0}
	Let $\langle S\rangle=\mathbb{Z}_n$ and $W\geq2$. Then in the group $M_{C_n(S,w)}^\ast$, the order of $\sum_{i=0}^{n-1}[v_i]$ divides $W-1$. Further if $\gcd(W-1,n)=1$ then order of  $\sum_{i=0}^{n-1}[v_i]$ is $W-1$.
\end{theorem}
\begin{proof}
	Let $S=\{s_1,s_2,\dots,s_k\}$. Then in $M_{C_n(S,w)}^*$, we have the following relations
$$[v_i]=\sum\limits_{s_j\in S} w(s_j)[v_{i+s_j}].$$

Let $\sigma=\sum_{i=0}^{n-1}[v_i]$. Then using the defining relations in $M_{C_n(S),w}^\ast$, we have 

\begin{center}
$\sigma=\sum\limits_{i=0}^{n-1}[v_i]
=\sum\limits_{i=0}^{n-1}\left( \sum\limits_{s_j\in S} w(s_j)[v_{i+s_j}]\right) 
=\sum\limits_{s_j\in S}w(s_j)\left(\sum\limits_{i=0}^{n-1}  [v_{i+s_j}]\right)$ \\
$=\sum\limits_{s_j\in S}w(s_j)\left( \sum\limits_{i=0}^{n-1} [v_{i}]\right)
=\left( \sum\limits_{s_j\in S}w(s_j)\right) \sigma=W\sigma.$
\end{center}
Thus, in the group $M_{C_n(S,w)}^*$, we have $\left( W-1\right) \sigma=0$. This proves the first part of the theorem.

By Theorem \ref{prop: K0 is cokernal}, $M_{C_n(S,w)}^*\cong \text{Coker}(I_n-A_{C_n(S,w)}^t)$, and under the isomorphism $[v_i]\mapsto \vec{b_i}+\text{Im}(I_n-A_{C_n(S,w)}^t)$, where $\vec{b_i}$ is the element of $\mathbb{Z}^n$ which has 1 in the $i^{th}$ coordinate and $0$ elsewhere. 

Hence for a natural number $d$, $d\sigma=0$ in $M_{C_n(S,w)}^*$ if and only if $d\vec{v}\in\text{Im}(I_n-A_{C_n(S,w)}^t)$ where $\vec{v}=(1,1,\dots,1)^t$. This is equivalent to $\vec{u}-A^t\vec{u}=d\vec{v}$ for some $\vec{u}=(u_0,u_1,\dots,u_{n-1})\in\mathbb{Z}^n$, which in turn equivalent to 
$$u_l-\sum\limits_{s_j\in S}w(s_j)u_{n-s_j+l}=d \hspace{1cm} 0\leq l\leq n-1 $$
Adding all the above equations we get
$$\sum\limits_{l=0}^{n-1}u_l-\sum\limits_{l=0}^{n-1}\sum\limits_{s_j\in S}w(s_j)u_{n-s_j+l}=nd$$

\begin{align*}
\text{LHS} &= \sum\limits_{l=0}^{n-1}u_l-\sum\limits_{s_j\in S}\sum\limits_{l=0}^{n-1}w(s_j)u_{n-s_j+l}\\
&= \sum\limits_{l=0}^{n-1}u_l-\sum\limits_{s_j\in S}w(s_j)\sum\limits_{l=0}^{n-1}u_{n-s_j+l}\\
&= \left( 1-\sum\limits_{s_j\in S}w(s_j)\right)\sum\limits_{l=0}^{n-1}u_l\\
&= \left( 1-W\right)\sum\limits_{l=0}^{n-1}u_l
\end{align*}

Thus $W-1$ divides $nd$. If $\gcd(W-1,n)=1$, then $W-1$ divies $d$. In particular, when $\gcd(W-1,n)=1$ order of $\sum\limits_{i=0}^{n-1}[v_i]$ is $W-1$.

\end{proof}

\begin{assumption}
In what follows, we always assume that $\langle S\rangle=\mathbb{Z}_n$ and $W\geq2$.
\end{assumption}

As we noted in \ref{lemma: eigenvalue of circulant matrix} for a circulant matrix $C$,
$$\det(C)=\prod\limits_{l=0}^{n-1}\left(\sum\limits_{j=0}^{n-1}c_j\zeta_n^{lj}\right)$$ 
where $\zeta_n=e^{\frac{2\pi i}{n}}$, the primitive $n^{th}$ root of unity. For $C_n(S,w)$, the adjacency matrix $A_{C_n(S,w)}$ is circulant. Also $I_n-A_{C_n(S,w)}^t$ is circulant (with integer entries). Let $S=\{s_1,s_2,\dots,s_k\}$. Then,
$$\det(I_n-A_{C_n(S)}^t)=\det(I_n-A_{C_n(S)})=\prod\limits_{l=0}^{n-1}\left(1-\sum\limits_{s_j\in S}w(s_j)\zeta_n^{ls_j}\right).$$ 

\begin{proposition}\label{prop: det is positive}
	Let $S_0:=\{j\in S\mid j\equiv0\pmod2\}$, $S_1:=\{j\in S\mid j\equiv1\pmod2\}$, $W_0:=\sum_{s_j\in S_0}w(s_j)$, and $W_1:=\sum_{s_j\in S_1}w(s_j)$. Then $\det(I_n-A_{C_n(S,w)}^t)>0$ if and only if $n$ is even and $1+W_1<W_0$.
\end{proposition}
\begin{proof}
	Let $P(x)=1-\sum_{s_j}w(s_j)x^{s_j}$ be the representer polynomial of $I_n-A_{C_n(S,w)}^t$. Let $z_l=P(\zeta_n^l)=1-\sum_{s_j}w(s_j)\zeta^{ls_j}$. It is easy to see that $z_0=1-\sum_{s_j\in S}w(s_j)=1-W<0$ and $z_{n-l}=\overline{z_l}$ for all $l$. Thus $\det(I_n-A_{C_n(S)}^t)>0$ if and only if $n$ is even and $z_{\frac{n}{2}}<0$. Since
	$$z_{\frac{n}{2}}=1-\sum\limits_{j\quad even}w(s_j)+\sum\limits_{j\quad odd}w(s_j),$$ Thus  $z_{\frac{n}{2}}<0$ iff $1+W_1<W_0$. 
\end{proof}
\begin{proposition}\label{prop: K_0 is infinite}
	Let $P(x)\in\mathbb{Z}[x]$ be the representer polynomial associated with the circulant matrix $I_n-A_{C_n(S,w)}^t$. Then $K_0(L(C_n(S,w)))$ is infinite if and only if  $P(x)$ and $x^n-1$ are relatively prime.
\end{proposition}
\begin{proof}
	Follows from Lemma \ref{lemma: circulant matrix is singular} and Proposition \ref{prop: K_0 of LPA}
\end{proof}

In order to compute the Grothendieck group of the Leavitt path algebra of $C_n(S,w)$, we look at the generating relations for $M_{C_n(S,w)}^\ast$
$$[v_i]=\sum\limits_{s_j\in S} w(s_j)[v_{i+s_j}].$$
where $0\leq i\leq n-1$, (subscripts are modulo $n$) and $S=\{s_1,s_2,\dots,s_k\}$, ($s_l<s_m$ for $l<m$). Any statement about $[v_0]$ in $M_{C_n(S,w)}^*$, has an analogous statement for $[v_k]$ for $0\leq k\leq n-1$, by symmetry of relations.

\begin{definition}\label{def: companion matrix}
	The \textit{companion matrix} of the monic polynomial $p(t)=c_0+c_1t+\dots+c_{n-1}t^{n-1}+t^n$, the $n\times n$ matrix defined as
	\[T(p)=\begin{pmatrix}
	0 & 0 & \dots & 0 & -c_0 \\
	1 & 0 & \dots & 0 & -c_1\\
	0 & 1 & \dots & 0 & -c_2\\
	\vdots & \vdots & \ddots & \vdots & \vdots \\
	0 & 0 & \dots & 1 & -c_{n-1} \\
	\end{pmatrix}\]
	
\end{definition}

Let a linear recursive sequence is of the form 
\begin{center}
	$u_{n+k}-c_{n-1}u_{n+k-1}-\dots-c_0u_n=0 \quad(n\geq0),$
\end{center} where $c_0,c_1,\dots,c_{n-1}$ are constants. The  \textit{characteristic polynomial} of the above linear recursive sequence is defined as $p(t)=t^n-c_{n-1}t^{n-1}-\dots-c_1t-c_0$ whose  companion matrix is
\[T(p)=\begin{pmatrix}
0 & 0 & \dots & 0 & c_0 \\
1 & 0 & \dots & 0 & c_1\\
0 & 1 & \dots & 0 & c_2\\
\vdots & \vdots & \ddots & \vdots & \vdots \\
0 & 0 & \dots & 1 & c_{n-1} \\
\end{pmatrix}\]

This matrix generates the sequence in the sense that,
\[
\begin{pmatrix}
a_k & a_{k+1} & \dots & a_{k+n-1} 
\end{pmatrix}
T(p)
=
\begin{pmatrix}
a_{k+1} & a_{k+2} & \dots & a_{k+n} 
\end{pmatrix}
\]

In particular, the $(n,n)^{th}$ entry of $T(p)^k$ is $u_{n+k-2}$.

When $0\notin S$, from the linear recursive relation in $M_{C_n(S,w)}^*$, we have the characteristic polynomial $p(S,w,t)=t^{s_k}-\sum_{s_j\in S}w(s_j)t^{s_k-s_j}$. The companion matrix of $p(S,w,t)$ is denoted by $T_{C_n(S,w)}$, is then the $s_k\times s_k$ matrix
\[T_{C_n(S,w)}=\begin{pmatrix}
0 & 0 & \dots & 0 &  \\
1 & 0 & \dots & 0 & \\
0 & 1 & \dots & 0 & \\
\vdots & \vdots & \ddots & \vdots &  \textbf{c}\\\\
0 & 0 & \dots & 1 &  \\
\end{pmatrix}\]
where \textbf{c} is the last column of $T_{C_n(S,w)}$ which contains entry $w(s_j)$ at positions $s_k-s_j+1$ and 0 elsewhere.

In $M_{C_n(S,w)}^*$, we observe that by writing the generating relations and then expanding the equation such that the subscripts are kept in \textit{increasing order}, at $i^{th}$ step we get the coefficients to be the last column of $T_{C_n(S,w)}^i$. 

The computation of the Smith Normal Form of $I_n-A_{C_n(S,w)}^t$ is the key tool for determining the $K_0$ of the Leavitt path algebra of $C_n(S,w)$. We show that this computation reduces to calculating the Smith Normal Form of an $s_k\times s_k$ matrix. 

\begin{theorem}\label{Coker}
 Let $n\in\mathbb{N}, S=\{s_1,s_2,\dots,s_k\}\subsetneq\mathbb{Z}_n$ such that $\langle S\rangle=\mathbb{Z}_n, 0\notin S, $ and $W\geq2$. Then \emph{Coker}$(I_n-A_{C_n(S,w)}^t)\cong \emph{Coker}(T_{C_n(S,w)}^n-I_n)$.
\end{theorem}
\begin{proof}
	Since the Smith normal form of $I_n-A_{C_n(S,w)}^t$ and $A_{C_n(S,w)}-I_n$ are the same, their cokernels are same and we only show that Coker$(A_{C_n(S,w)}-I_n)\cong$ Coker$(T_{C_n(S,w)}^n-I_n)$. For simplicity we write $B=A_{C_n(S,w)}-I_n$ and $T=T_{C_n(S,w)}$. First we observe that 
	\begin{center}
		\[   
		B_{pq} = 
		\begin{cases}
		-1 &\quad\text{if}\quad q=p\\
		w(s_j) &\quad\text{if}\quad q=p+s_j\\
		0 &\quad\text{otherwise}\\
		\end{cases}
		\]
	\end{center}
Let $P$ be a $(s_k\times s_k)$ lower triangular matrix given by
\begin{center}
	\[   
	P_{pq} = 
	\begin{cases}
	0 &\quad\text{if}\quad q>p\\
	w(s_j) &\quad\text{if}\quad p-q=s_k-s_j\\
	0 &\quad\text{otherwise}\\
	\end{cases}
	\]
\end{center}
and let $Q$ be a $(s_k\times s_k)$ upper triangular matrix given by
\begin{center}
	\[   
	Q_{pq} = 
	\begin{cases}
	0 &\quad\text{if}\quad q<p\\
	-1 &\quad\text{if}\quad q=p\\
	w(s_j) &\quad\text{if}\quad q-p=s_j\\
	0 &\quad\text{otherwise}\\
	\end{cases}
	\]
\end{center}

It is direct that $P$ and $Q$ are invertible. Let $R=-Q^{-1}$. Then a direct computation yields $PR=T^{s_k}$, and also $QR=-I_{s_k}$. Let $P^\prime$ be the block matrix $\left[P\mid 0_{s_k\times(n-s_k)}\right]$ and $Q^\prime$ be the block matrix $\left[0_{s_k\times(n-s_k)}\mid Q \right]$. The $(s_k\times(n-s_k))$ submatrix of $B$ consisting of bottom $s-k$ rows can be written as $P^\prime+Q^\prime$.

The first $(n-s_k)$ reduction steps of the Smith normal form will result in an $(n-s_k)\times(n-s_k)$ identity submatrix in the upper left corner. On the bottom $s_k$ rows, the $i^{th}$ reduction step adds the $i^{th}$ column to the sum of $w(s_j)$ times $(i+s_k)^{th}$ columns, then zeros out the $i^{th}$ column. The matrix that accomplishes this reduction step is
\[P^{-1}TP=\begin{pmatrix}
 &  & \textbf{r} &  &  &   \\
1 & 0 & 0 & \dots & 0 & 0\\
0 & 1 & 0 & \dots & 0 & 0\\
\vdots & \vdots & \vdots & \ddots & \vdots & \vdots\\
0 & 0 & 0 & \dots & 1 & 0\\
\end{pmatrix}\]
where \textbf{r} is the first row contains entry $w(s_j)$ at positions $s_k$ and $0$ elsewhere. After $i$ reduction steps, the first $(s_k\times s_k)$ submatrix with nonzero column vectors on the bottom $s_k$ rows will be
$$P\cdot(P^{-1}TP)^i=T^iP.$$

Therefore the first $(n-s_k)$ reduction steps of the Smith Normal Form will result in the following form.
	\[B\sim\begin{pmatrix}		
	I_{(n-s_k)} &  0_{(n-s_k)\times s_k} & \\
	
	0_{s_k\times(n-s_k)} &  T^{n-s_k}P+Q
	\end{pmatrix}
	\]
Since $(T^{n-s_k}P+Q)R=T^{n}-I_{s_k}$,
\[B\sim\begin{pmatrix}		
I_{(n-s_k)} &  0_{(n-s_k)\times s_k} & \\

0_{k\times(n-s_k)} &  T^{n}-I_{s_k}
\end{pmatrix}
\]	
Hence Coker$(B)\cong$Coker$(T^n-I_{s_k})$.\\

\end{proof} 

\section{Illustrations}
As illustrations of the above discussion we consider some simple cases when $W\geq 2$ and $\sum\limits_{i=0}^{n-1}[v_i]$ is the identity in $M_{C_n(S,w)}^*$, which recovers the examples obtained in  \cite{MR3307385},\cite{MR3456905},and \cite{MR3854343}.

\subsection{$S=\mathbb{Z}_n$}\hfill\\

In this subsection we only look at the following two simple cases when $\sum\limits_{i=0}^{n-1}[v_i]$ is the identity in $M_{C_n(\mathbb{Z}_n,w)}^\ast$

\begin{definition}\label{def:K_n}
	Let $n,l$ be two positive integers. We define $K_n^{(l)}$ to be the graph with $n$ vertices $v_0,v_1,\dots v_{n-1}$,  in which there is exactly one edge from $v_i$ to $v_j$ for each $0\leq i\neq j\leq n-1$ and $l$ loops at each vertex. We call $K_n^{(l)}$ the \textbf{complete $n$-graph with $l$ loops}.
\end{definition}

\begin{theorem}\label{theorem: complete graph}
Let $n\geq 2$ be a positive integer.
    \begin{enumerate}
        \item  $L(K_n^{(1)})\cong L(1,n)$.
        \item Let $E$ be a finite graph such that $L(E)$ is purely infinite simple. If $K_0(L(E))\cong\mathbb{Z}^n$ and $[L(E)]$ is identity in $K_0(L(E))$, then $L(E)\cong L(K_{n+1}^{(2)})$.
        \end{enumerate}
\end{theorem}
\begin{proof}
Let $w_l:S\rightarrow \mathbb{N}$ be the weight function defined by $w_l(0)=l$ and $w_l(i)=1$ for $1\leq i\leq n-1$. Then it is direct that  $C_n(\mathbb{Z}_n,w_l)\cong K_n^{(l)}$.
1. We note that 
$$\det(I_n-A_{K_n^{(1)}}^t)=\prod\limits_{l=0}^{n-1}(-1)(\sum\limits_{j=1}^{n-1}\zeta^{lj})=-(n-1)<0.$$
Also we have $W-1=|S|-1=n-1$. So $\gcd(W-1,n)=1$ and hence $\sum\limits_{i=0}^{n-1}[v_i]$ is the identity in $M_{K_n^{(1)}}^\ast$. Also determinant divisors theorem yields that 
$$\text{SNF}(I_n-A_{K_n^{(1)}}^t)=\text{diag}(1,1,\dots,1,n-1).$$
Hence,
$K_0(L(K_n^{(1)}))\cong\mathbb{Z}_{n-1}.$
By Proposition \ref{prop: Leavitt algebra}, the result follows.

2. We note that $(I_n-A_{K_n^{(2)}}^t)$ is the $n\times n$ matrix with every entry $-1$. Hence $\det(I_n-A_{K_n^{(2)}}^t)=0$ and rank($I_n-A_{K_n^{(2)}}^t)=1$. Therefore if $n\geq2$, then $$K_0(L(K_n^{(2)}))\cong\mathbb{Z}^{n-1}.$$
Also in $K_0(L(K_n^{(2)}))$, $$\sigma=\sum\limits_{i=0}^{n-1}[v_i]=[v_0]+\sum\limits_{i=1}^{n-1}[v_i]=\left( 2[v_0]+\sum\limits_{i=1}^{n-1}[v_i]\right) +\sum\limits_{i=1}^{n-1}[v_i]=2\sum\limits_{i=0}^{n-1}[v_i]=2\sigma.$$
Hence $\sum\limits_{i=0}^{n-1}[v_i]$ is the identity in $K_0(L(K_n^{(2)}))$.
Applying Algebraic KP Theorem we have the result.
\end{proof}

\subsection{$|S|=1$}\hfill\\

Let $S=\{i\}$. Since $\left\langle S\right\rangle =\mathbb{Z}_n$, $\gcd(i,n)=1$ and the weight function $w:S\rightarrow\mathbb{N}$ is given by $w(i)=W$. Let $D_n^k$ be the graph with $n$ vertices $v_0, v_1, \dots, v_n$ and $kn$ edges such that every vertex $v_i$ emit $k$ edges to $v_{i+1}$. We call $D_n^k$ an $k$-cycle of length $n$. 

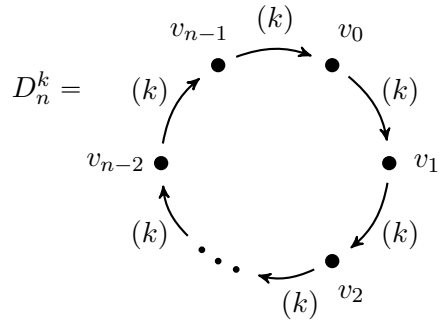
\begin{center}
	\begin{tikzpicture}
	[->,>=stealth',shorten >=1pt,thick]

	\draw [fill=black] (0:1.5)		circle [radius=0.08];
	\draw [fill=black] (60:1.5)		circle [radius=0.08];
	\draw [fill=black] (120:1.5)		circle [radius=0.08];
	\draw [fill=black] (180:1.5)		circle [radius=0.08];
	\draw [fill=black] (-60:1.5)		circle [radius=0.08];
	\draw [fill=black] (230:1.5)		circle [radius=0.03];
	\draw [fill=black] (240:1.5)		circle [radius=0.03];
	\draw [fill=black] (250:1.5)		circle [radius=0.03];
	
	\path[->]	(110:1.5)	edge[bend left=20] 	node	{}		(70:1.5);
	\path[->]	(50:1.5)	edge[bend left=20] 	node	{}		(10:1.5);
	\path[->]	(-10:1.5)	edge[bend left=20] 	node	{}		(-50:1.5);	
	\path[->]	(-70:1.5)	edge[bend left=20] 	node	{}		(260:1.5);
	\path[->]	(220:1.5)	edge[bend left=20] 	node	{}		(190:1.5);
	\path[->]	(170:1.5)	edge[bend left=20] 	node	{}		(130:1.5);

	\node at (60:2) {$v_0$};
	\node at (0:2) {$v_1$};
	\node at (-60:2) {$v_2$};
	\node at (120:2) {$v_{n-1}$};
	\node at (180:2.1) {$v_{n-2}$};
	\node at (-3,1) {$D_n^k =$};
	\node at (90:1.9) {$(k)$};
	\node at (30:1.9) {$(k)$};
	\node at (150:1.9) {$(k)$};
	\node at (-30:1.9) {$(k)$};
	\node at (-80:1.9) {$(k)$};
	\node at (-150:1.9) {$(k)$};

	\end{tikzpicture}
\end{center}
It is easy to see that $C_n(S,w)\cong D_n^W$.

The generating relations for $M_{D_n^W}^\ast$ are given by  $$[v_i]=W[v_{i+1}]$$ for $0\leq i\leq n$, where the subscripts are interpreted mod $n$. So for each $0\leq i\leq n$ we have that $$[v_i]=W[v_{i+1}]=W^2[v_{i+2}]=\dots =W^{n-i}[v_{n-1}]=W^{n+1-i}[v_0].$$ In particular, each $[v_i]$ is in the subgroup of $M_{D_n^W}^\ast$ generated by $[v_0]$. Since the set $\{[v_i]\mid 0\leq i\leq n-1\}$ generates $M_{D_n^W}^\ast$, we conclude that $M_{D_n^W}^\ast$ is cyclic, and $[v_0]$ is a generator. 

We also observe that $$\det(I_n-A_{D_n^W}^t)=\prod\limits_{l=0}^{n-1}(1-W\zeta^l)=1-W^n<0.$$
We conclude that $|K_0(L(D_n^W))|=W^n-1$. Thus we have $$K_0(L(C_n(S,w)) \cong M_{D_n^W}^\ast \cong\mathbb{Z}_{W^n-1}.$$

\begin{proposition}\label{prop: S is singleton}
	Let $S=\{i\},\gcd(i,n)=1,$ and $\gcd(W-1,n)=1$, then  $$L(C_n(S,w))\cong M_{W^n-1}(L(1,W^n)).$$ 
\end{proposition}
\begin{proof}
	$\sum_{i=0}^{n-1}[v_i]$ is the identity in the group $M_{C_n(S,w)}^\ast$. Hence by Proposition \ref{prop: Leavitt algebra} the result follows.
\end{proof}
\begin{corollary}(\cite{MR3456905}, Proposition 3.4)
	Assume the hypothesis of Proposition \ref{prop: S is singleton} and $W=2$, then $L(C_n(S,w))\cong M_{2^n-1}(L(1,2^n))$
\end{corollary}

\subsection{$|S|=2$}\hfill\\

Let $S=\{s_1,s_2\}$ with $s_1<s_2$. Let $a,b\in\mathbb{N}$. We define $w(s_1):=a$ and $w(s_2):=b$. Thus $W=a+b\geq 2$. Since $\left\langle S\right\rangle =\mathbb{Z}_n$, it is sufficient to consider only the following subcases: 
\begin{enumerate}
	\item $s_1=0$ and $s_2=1$.
	\item $s_1=1$.
	\item $s_1$ and $s_2$ divide $n$ with $1<s_1<s_2$,  and $\gcd(s_1,s_2)=1$.
\end{enumerate}
In what follows we consider these subcases separately.

\begin{lemma}\label{lemma: sigma is identity when S has two elements}
	In each of the above subcases if $a=b=1$, then $\sum\limits_{i=0}^{n-1}[v_i]$ is the identity in $M_{C_n(S,w)}^*$.
\end{lemma}
\begin{proof}
	Since $W-1=|S|-1=1$ in these subcases we have $\gcd(W-1,n)=1$ and the result follows from Theorem \ref{theorem: order of identity in K_0}.
\end{proof}

\begin{proposition}\label{prop:det(a,b)=0}
	Let $n,a,b\in\mathbb{N}$ be fixed. Let $0\leq s_1<s_2\leq n-1$. Consider the w-Cayley graph $C_n(S,w)$ where $S=\{s_1,s_2\}$ and $w(s_1)=a,w(s_2)=b$. Then $\emph{det}(I_n-A_{C_n(S,w)}^t)=0$ if and only if exactly one of the following occurs:	
	\begin{enumerate}
		\item $a=b=1, n\equiv0\pmod6, s_2\equiv5s_1\pmod6$.
		\item $a=b+1$, $n$ is even, $s_1$ is even, $s_2$ is odd.
		\item $b=a+1$, $n$ is even, $s_1$ is odd, $s_2$ is even.
	\end{enumerate}
\end{proposition} 
\begin{proof}
	Let $\Delta=\det(I_n-A_{C_n(S,w)}^t)$ and $z_l=a\zeta^{ls_1}+b\zeta^{ls_2}$. Since $$\Delta=\prod\limits_{l=0}^{n-1}\left( 1-a\zeta^{ls_1}-b\zeta^{ls_2}\right) $$ Then $\Delta=0$ if and only if $z_l=1$ for some $l$. We observe that $z_0=a+b>1$ and $z_{n-l}=\overline{z_l}$. So we can write 
	
	$$\Delta =
	\left\{
	\begin{array}{ll}
	(1-z_0)\prod\limits_{l=1}^{\frac{n-1}{2}}(1-z_l)(\overline{1-z_l})  & \mbox{if } n\quad\text{is odd} \\
	
	(1-z_0)(1-a(-1)^{s_1}-b(-1)^{s_2})\prod\limits_{l=1}^{\frac{n}{2}-1}(1-z_l)(\overline{1-z_l})  & \mbox{if } n\quad\text{is even}
	\end{array}
	\right.$$
	Hence we can assume $1\leq l\leq [\frac{n}{2}]$, where $[\frac{n}{2}]$ is the integer part of $\frac{n}{2}$. Further, $z_l=1$ implies $\left| a\zeta^{ls_1}+b\zeta^{ls_2}\right| =1$. So $1=\left| a\zeta^{ls_1}+b\zeta^{ls_2}\right|\geq\left||a|-|b| \right|=\left| a-b\right| \geq0$. Since $a,b\in \mathbb{N}$, only possiblities are $a=b, a=b+1,$ or $b=a+1$.\\
	
	\textbf{Case 1:} $a=b$\\
	
	Let $\theta=\frac{2\pi l}{n}$. Then $z_l=1$ if and only if
	$$a\left( \cos s_1\theta + \cos s_2\theta\right)=1\quad\text{and}\quad a\left( \sin s_1\theta+\sin s_2\theta\right) =0.$$
	The second equation implies that $s_1\theta\equiv -s_2\theta \pmod 2\pi$. Substituting back in first equation we get, 
	$$1=a\left( \cos (-s_2\theta)+\cos s_2\theta\right) =2a\cos s_2\theta$$
	$$\Rightarrow\cos s_2\theta=\frac{1}{2a} \Rightarrow
	\frac{2\pi l}{n} s_2=\arccos\left( \frac{1}{2a}\right).$$ 
	Thus $n\in\mathbb{N}$ only if $a=1$. Assuming $a=1$, we have 
	$\arccos\frac{1}{2}=\frac{\pi}{3}$ or $\frac{5\pi}{3}.$
	Substituting back, we see that
	$$\frac{2\pi ls_2}{n}=\frac{\pi}{3}\Rightarrow n=6s_2l,\quad\text{or}\quad\frac{2\pi ls_2}{n}=\frac{5\pi}{3}\Rightarrow 5n=6s_2l.$$
	In either case, $n\equiv0\pmod6$. Also, $s_2\theta\equiv-s_1\theta\pmod 2\pi$ implies that for some integer $m$,
	$$(s_2+s_1)\frac{\pi}{3}=2\pi m\Rightarrow s_2+s_1=6m\quad\text{or}\quad(s_2+s_1)\frac{5\pi}{3}=2\pi m\Rightarrow 5(s_2+s_1)=6m.$$
	In either case, $s_2+s_1\equiv0\pmod6$, or $s_2\equiv5s_1\pmod6$.
	
	Conversely, when $a=1, n\equiv0\pmod6$ and $s_2\equiv5s_1\pmod6$, then letting $l=6$ implies that
	$$z_l=\omega^{ls_1}+\omega^{ls_2}=\left( e^{\frac{2\pi i}{6}}\right)^{s_1}+\left( e^{\frac{2\pi i}{6}}\right)^{-s_1}=1$$
	
	\textbf{Case 2:} $a=b+1$\\
	
	As in case 1, let $\theta=\frac{2\pi l}{n}$. Then $z_l=1$ if and only if 
	$$(b+1)\cos s_1\theta+b\cos s_2\theta=1\quad \text{and}\quad (b+1)\sin s_1\theta+b\sin s_2\theta=0.$$
	The second equation implies that $s_1\theta=\arcsin\left( \frac{-b}{b+1}\sin s_2\theta\right)$. Substituting back in the first equation we get,
	$$(b+1)\cos \left(\arcsin\left( \frac{-b}{b+1}\sin s_2\theta\right) \right) +b\cos s_2\theta=1.$$
	Since $\cos(\arcsin x)=\sqrt{1-x^2}$, we have
	$$(b+1)\sqrt{1-\left( \frac{b}{b+1}\sin s_2\theta\right)^2}+b\cos s_2\theta=1.$$ Hence,
	$$\sqrt{b^2+2b+1-b^2\sin^2s_2\theta}=1-b\cos s_2\theta.$$
	Squaring both sides,
	$$b^2\cos^2s_2\theta+2b+1=b^2\cos^2s_2\theta-2b\cos s_2\theta+1
	\Rightarrow \cos s_2\theta=-1.$$
	Therefore, $s_2\theta\equiv\pi\pmod 2\pi$. Substituting $\theta=\frac{2\pi l}{n}$, we see that $n$ is even. Also, $s_2\theta\equiv\pi\pmod 2\pi$ implies that $(s_2-1)\pi=2\pi m$ for some integer $m$. So, $s_2=2m+1$ or $s_2$ is odd. Also since, $s_1\theta=\arcsin\left( \frac{-b}{b+1}\sin s_2\theta\right)=\arcsin(0)$, $s_1\pi=0$ or $\pi$. $(b+1)\cos s_1\pi-b=1\Rightarrow s_1\pi\equiv0\pmod 2\pi$. Hence $s_1$ is even.
	
	Conversely, let $n,s_1$ be even and $s_2$ be odd then by taking $l=\frac{n}{2}$, we get
	$$z_l=(b+)\omega_l^{s_1}+b\omega_l^{s_2}=(b+1)(-1)^{s_1}+b(-1)^{s_2}=b+1-b=1.$$
	
	\textbf{Case 3:} $b=a+1$
	
	The proof is similar to that of case 2.
	
\end{proof}

\begin{corollary}\label{K0(a,b)=0}
	Assume the hypothesis of Proposition \ref{prop:det(a,b)=0}. Further assume that $L(C_n(S,w))$ is unital purely infinite simple. Then $K_0(L(C_n(S,w)))$ is infinite abelian group if and only if one of the following holds:
	\begin{enumerate}
		\item $a=b=1, n\equiv0\pmod6, s_2\equiv5s_1\pmod6$.
		\item $a=b+1$, $n$ is even, $s_1$ is even, $s_2$ is odd.
		\item $b=a+1$, $n$ is even, $s_1$ is odd, $s_2$ is even.
	\end{enumerate}
In which case $\emph{rank}(K_0(L(C_n(S,w))))=n-\emph{rank}(I_n-A_{C_n(S,w)}).$	
\end{corollary}

\subsubsection{\textbf{Subcase} 2.1: $S=\{0,1\}$}\hfill\\

Let $F_n^{(a,b)}$ be the graph with $n$ vertices $v_0,v_1,\dots, v_{n-1}$ and $ak+bk$ edges such that at every vertex $v_l$, there are $a$ loops and $b$ edges getting emitted into $v_{l+1}$ (subscripts are mod $n$).  

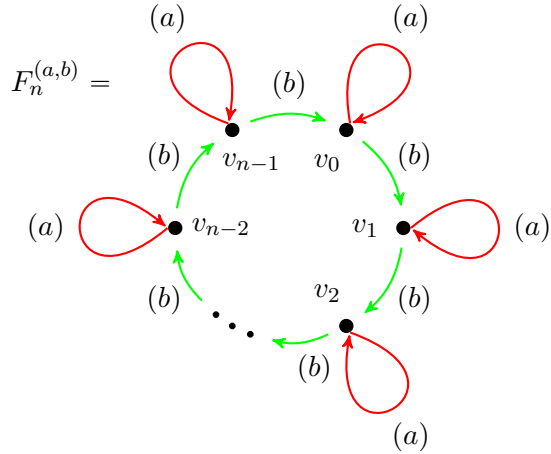
\begin{center}
	\begin{tikzpicture}
	[->,>=stealth',shorten >=1pt,thick]

	\draw [fill=black] (0:1.5)		circle [radius=0.08];
	\draw [fill=black] (60:1.5)		circle [radius=0.08];
	\draw [fill=black] (120:1.5)		circle [radius=0.08];
	\draw [fill=black] (180:1.5)		circle [radius=0.08];
	\draw [fill=black] (-60:1.5)		circle [radius=0.08];
	\draw [fill=black] (230:1.5)		circle [radius=0.03];
	\draw [fill=black] (240:1.5)		circle [radius=0.03];
	\draw [fill=black] (250:1.5)		circle [radius=0.03];
	
	\path[->]	(110:1.5)	edge[green, bend left=20] 	node	{}		(70:1.5);
	\path[->]	(50:1.5)	edge[green, bend left=20] 	node	{}		(10:1.5);
	\path[->]	(-10:1.5)	edge[green, bend left=20] 	node	{}		(-50:1.5);	
	\path[->]	(-70:1.5)	edge[green, bend left=20] 	node	{}		(260:1.5);
	\path[->]	(220:1.5)	edge[green, bend left=20] 	node	{}		(190:1.5);
	\path[->]	(170:1.5)	edge[green, bend left=20] 	node	{}		(130:1.5);
	
	\path[->] (0:1.6) edge [red, out=40, in=320, looseness=4,  distance=2cm, ->] node {} (0:1.6);
	\path[->] (60:1.6) edge [red, out=100, in=20, looseness=4, distance=2cm, ->] node {} (60:1.6);
	\path[->] (120:1.6) edge [red, out=160, in=80, looseness=4,  distance=2cm, ->] node {} (120:1.6);
	\path[->] (180:1.6) edge [red, out=220, in=140, looseness=4,  distance=2cm, ->] node {} (180:1.6);
	\path[->] (-60:1.6) edge [red, out=-20, in=260, looseness=4,  distance=2cm, ->] node {} (-60:1.6);

	\node at (60:1) {$v_0$};
	\node at (0:1) {$v_1$};
	\node at (-60:1) {$v_2$};
	\node at (120:1) {$v_{n-1}$};
	\node at (180:0.9) {$v_{n-2}$};
	\node at (-3,2) {$F_n^{(a,b)} =$};
	
	\node at (90:1.9) {$(b)$};
	\node at (30:1.9) {$(b)$};
	\node at (150:1.9) {$(b)$};
	\node at (-30:1.9) {$(b)$};
	\node at (-80:1.9) {$(b)$};
	\node at (-150:1.9) {$(b)$};
	
	\node at (0:3.2) {$(a)$};
	\node at (60:3.2) {$(a)$};
	\node at (120:3.2) {$(a)$};
	\node at (180:3.2) {$(a)$};
	\node at (-60:3.2) {$(a)$};

	\end{tikzpicture}
\end{center}
Then $C_n(S,w)\cong F_n^{(a,b)}$ when $S=\{0,1\}$. We note that
$$\text{det}(I_n-A_{F_n^{(a,b)}}^t) =\prod\limits_{l=0}^{n-1}(1-a-b\zeta^l)=(1-a)^n-b^n.$$
\begin{lemma}\label{lemma: det is non-negative when S={0,1}}
	Let $n,a,b\in \mathbb{N}$. Then
	\begin{center}
		$\det(I_n-A_{F_n^{(a,b)}}^t)\geq 0$ if and only if $n$ is even and $a\geq b+1$.
	\end{center}
	Moreover, $\det(I_n-A_{F_n^{(a,b)}}^t)=0$ if and only if $n$ is even and $a=b+1$.
	
\end{lemma}
\begin{proof}
	We refer to the proof of Proposition \ref{prop:det(a,b)=0}. We need to substitute $s_1=0$, and $s_2=1$. Since,
	$$\Delta =
	\left\{
	\begin{array}{ll}
	(1-z_0)\prod\limits_{l=1}^{\frac{n-1}{2}}(1-z_l)(\overline{1-z_l})  & \mbox{if } n\quad\text{is odd} \\
	
	(1-z_0)(1-a(-1)^j-b(-1)^k)\prod\limits_{l=1}^{\frac{n}{2}-1}(1-z_l)(\overline{1-z_l})  & \mbox{if } n\quad\text{is even}
	\end{array}
	\right.$$
	$\Delta\geq 0$ if and only if $n$ is even and $a\geq b+1$, in which case $1-z_{\frac{n}{2}}=1-a+b\leq 0$. Also, it follows that, det$(I_n-A_{F_n^{(a,b)}}^t)=0$ if and only if $n$ is even and $a=b+1$. 
	
\end{proof}

We describe the Smith Normal Form of $I_n-A_{F_n^{(a,b)}}^t$.
\begin{lemma}\label{lemma: SNF when S={0,1}}
	Suppose $n\in\mathbb{N}$. Let $T$ be the $n \times n$ circulant matrix whose first row is $\vec{t}=((1-a), -b, 0, \dots, 0)$.Let $\gcd(1-a,b)=d$. Then the Smith Normal Form
	$$\emph{SNF}(T)=\emph{diag}\left( d,d,\dots, d,\frac{|(1-a)^n-b^n|}{d^{n-1}}\right) $$
\end{lemma}
\begin{proof}
	In order to compute Smith Normal Form of $T$, we use at the determinant divisors theorem  and look at $i \times i$ minors of $T$ for each $1\leq i\leq n$. Let $\alpha_i$ be the gcd of the set of all $i\times i$ minors of $T$ and $\alpha_0=1$. Then 
	$$\text{SNF}(T)=\text{diag}\left(\frac{\alpha_1}{\alpha_0},\frac{\alpha_2}{\alpha_1},\dots,\frac{|\det(T)|}{\alpha_{n-1}} \right) .$$
	
	By the definition of $T$, it is easy to observe that $\alpha_i=\gcd\left( (a-1)^i,b^i\right)=\gcd(a,b)^i=d^i$ for $1\leq i\leq n-1$ and $|\det(T)|=|(1-a)^n-b^n|$.
	Therefore 
	
	\begin{align*}
	\text{SNF}(T) &= \text{diag}\left(\frac{d}{1},\frac{d^2}{d},\frac{d^3}{d},\dots,\frac{|(1-a)^n-b^n|}{d^{n-1}} \right)\\
	&= \text{diag}\left(d,d,d,\dots,\frac{|(1-a)^n-b^n|}{d^{n-1}} \right)
	\end{align*}
\end{proof} 
\begin{theorem}\label{theorem: S={0,1}}
	Let $n,a,b\in\mathbb{N}$ be fixed. Suppose $S=\{0,1\}\subset\mathbb{Z}_n$, $w:S\rightarrow\mathbb{N}$ be defined as $w(0)=a$ and $w(1)=b$. Let $d=\gcd(a-1,b)$. Then

	$$K_0(L(C_n(S,w)))\cong
	\left\{
	\begin{array}{ll}
	\left(\mathbb{Z}_d\right) ^{n-1} \oplus \mathbb{Z}  & \mbox{if } a=b+1 \quad\text{and}\quad n\quad\text{is even}\\
	\left(\mathbb{Z}_d\right) ^{n-1} \oplus \mathbb{Z}_{\frac{|(1-a)^n-b^n|}{d^{n-1}}}  & \mbox{otherwise} 
	
	\end{array}
	\right.$$ 
\end{theorem}
\begin{proof}
	Follows from the above lemmas \ref{lemma: det is non-negative when S={0,1}} and \ref{lemma: SNF when S={0,1}}.
\end{proof}

\begin{example}
	$L(C_n(0,1))\cong L(1,2)$.	
\end{example}
\begin{proof}
	In Theorem \ref{theorem: S={0,1}} we take $a=b=1$. Then $\gcd(a-1,b)=1$. Hence $\det(I_n-A_{C_n(0,1)}^t)=-1<0$ and  $K_0(L(C_n(0,1)))$ is trivial.
	By Proposition \ref{prop: Leavitt algebra} we have $L(C_n(0,1))\cong L(1,2)$.
\end{proof}
The above example was observed in \cite{MR3456905}, Proposition 3.3.

\subsubsection{\textbf{Subcase 2.2}: $S=\{1,j\}$ with $j>1$}\hfill\\

We note that by Proposition \ref{prop: det is positive} $\det(I_n-A_{C_n(S,w)}^t)>0$ if and only if $n,j$ are even and $b>a+1$. Also by Proposition \ref{prop:det(a,b)=0} $\det(I_n-A_{C_n(S,w)}^t)=0$ if and only if one of the following occurs:
\begin{enumerate}
	\item $a=b=1, n\equiv0\pmod6,j\equiv5\pmod6$
	\item $b=a+1, n, j$ are even.
\end{enumerate}
In order to compute $K_0(L(C_n(S,w)))$, we apply Theorem \ref{Coker} and compute the Smith normal form of $T_{C_n(S,w)}^n-I_n$. This procedure is performed for unweighted Cayley graph in \cite{MR3854343}. However, the record an interesting example here. 

\subsubsection{\textbf{Leavitt Path algebras of Cayley graphs of Dihedral groups}}\hfill\\

Let $\tilde{D}_n$ be the dihedral group of order $2n$. i.e. $\tilde{D}_n=\left\langle r,s\mid r^n=s^2=e,rsr=s\right\rangle$. Let $D_n$ denote the Cayley graph of $\tilde{D}_n$ with respect to the generating subset $S=\{r,s\}$.

The following discussion is taken from \cite{MR2785945}. A graph transformation is called standard if it is one of the following types: in-splitting, in-amalgamation, out-splitting, out-amalgamation, expansion, or contraction. For definitions the reader to referred to \cite{MR3729290}. If $E$ and $F$ are graphs having no sources and no sinks, a flow equivalence from $E$ to $F$ is a sequence $E=E_0\rightarrow E_1\rightarrow\dots\rightarrow E_n=F$ of graphs and standard graph transformations which starts at $E$ and ends at $F$.

\begin{proposition}\cite[Corollary 6.3.13]{MR3729290} \label{prop: flow equivalence}
	Suppose $E$ and $F$ are finite graphs with no sources whose corresponding Leavitt path algebras are purely infinite simple. Then $E$ is flow equivalent to $F$ if and only if $\det(I_{|E|}-A_E)=\det(I_{|F|}-A_F)$ and $\emph{Coker}(I_{|E|}-A_E)\cong\emph{Coker}(I_{|F|}-A_F)$.
\end{proposition}

\begin{definition}[In-splitting]
	Let $E=(E^0,E^1,r,s)$ be a directed graph. For each $r^{-1}(v)\neq\phi$, partition the set $r^{-1}(v)$ into disjoint nonempty subsets $\mathcal{E}_1^v,\dots,\mathcal{E}_{m(v)}^v$ where $m(v)\geq1$. If $v$ is a source then set $m(v)=0$. Let $\mathcal{P}$ denote the resulting partition of $E^1$. We form the in-split graph $E_r(\mathcal{P})$ from $E$ using the partition $\mathcal{P}$ as follows:
	\begin{center}
	$E_r(\mathcal{P})^0=\{v_i\mid v\in E^0,1\leq i\leq m(v)\}\cup\{v\mid m(v)=0\}$,\\
	$E_r(\mathcal{P})^1=\{e_j\mid e\in E^1,1\leq j\leq m(s(e))\}\cup\{e\mid m(s(e))=0\},$
	\end{center}
and define $r_{E_r(\mathcal{P})},s_{E_r(\mathcal{P})}:E_r(\mathcal{P})^1\rightarrow E_r(\mathcal{P})^0$ by
\begin{center}
	$s_{E_r(\mathcal{P})}(e_j)=s(e)_j$ and $s_{E_r(\mathcal{P})}(e)=s(e)$\\
	$r_{E_r(\mathcal{P})}(e_j)=r(e)_i$ and $s_{E_r(\mathcal{P})}(e)=s(e)_i$ where $e\in\mathcal{E}_i^{r(e)}.$\\
\end{center}
\end{definition}

We observe that $D_n$ can be obtained from $C_n^{n-1}$ by the standard operation in-splitting with respect to the partition $\mathcal{P}$ of the edge set of $C_n^{n-1}$ that places each edge in its own singleton partition class. In \cite{MR3307385} the collection of Leavitt path algebras $\{L(C_n^{n-1})\mid n\in\mathbb{N}\}$ is completely described and by Proposition \ref{prop: flow equivalence} we have that the same description holds true if we replace $C_n^{n-1}$ with $D_n$ for every $n\in\mathbb{N}$. Hence we have
\begin{theorem}\label{theorem: Dihedral group}
	For each $n\in\mathbb{N}$, $\det(I_n-A_{D_n}^t)\leq0$. And
	\begin{enumerate}
		\item If $n\equiv1~\text{or}~5 \pmod6$ then $K_0(L(D_n))\cong\{0\}$ and $L(D_n)\cong L(1,2)$.
		
		\item If $n\equiv2~\text{or}~4 \pmod6$ then $K_0(L(D_n))\cong\mathbb{Z}/3\mathbb{Z}$ and $L(D_n)\cong M_3(L(1,4))$.
		
		\item If $n\equiv3 \pmod6$ then $K_0(L(D_n))\cong\left(\mathbb{Z}/2\mathbb{Z} \right)^2 $
		
		\item If $n\equiv0 \pmod6$ then $K_0(L(D_n))\cong\mathbb{Z}^2$ and $L(D_n)\cong L(K_3^{(2)})$ 
	\end{enumerate}
\end{theorem}

\subsubsection{\textbf{Subcase 2.3}: $S=\{s_1,s_2\}$ where  $s_1,s_2$ divide $n$, $1<s_1<s_2$ and $\gcd(s_1,s_2)=1$}\label{subsection S={i,j}}\hfill\\

By Proposition \ref{prop:det(a,b)=0} and by Proposition \ref{prop: det is positive}, we have that $\det(I_n-A_{C_n(S,w)}^t)=0$ if and only if one of the following occurs:
\begin{enumerate}
	\item $a=b=1,n\equiv0\pmod6,d_2\equiv5d_1\pmod6.$
	\item $a=b+1$, $n,d_1$ are even, $d_2$ is odd.
	\item $b=a+1$, $n,d_2$ are even, $d_1$ is odd.
\end{enumerate}
and $\det(I_n-A_{C_n(S,w)}^t)>0$ if and only if one of the following occurs:
\begin{enumerate}
	\item $a>b+1$, $n,d_1$ are even, $d_2$ is odd.
	\item $b>a+1$, $n,d_2$ are even, $d_1$ is odd.	
\end{enumerate}
In order to compute $K_0(L(C_n(S,w)))$, we apply Theorem \ref{Coker} and compute the Smith Normal form of $T_{C_n(S,w)}^n-I_n$.

We illustrate this when $S=\{d_1,d_2\},$ where $d_1,d_2$ divides $n$, $\gcd(d_1,d_2)=1$ and $a=b=1$.
In this special case we have $\det(I_n-A_{C_n(d_1,d_2)}^t)=0$ if and only if $n\equiv 0\pmod 6$ and $d_2\equiv5d_1\pmod6$. In all other cases, we have $\det(I_n-A_{C_n^k}^t)<0$. Define $H_{(d_1,d_2)}(n):=|\det(I_n-A_{C_n^k}^t)|$. In order to compute $K_0(L(C_n(d_1,d_2)))$, we apply Theorem \ref{Coker} and compute the Smith normal form of $T_{C_n(d_1,d_2)}^n-I_n$.
	
	For $1\le j,k\in\mathbb{N}$ let us define a sequence $F_{(j,k)}$ recursively as follows:
	\begin{center}
		\[   
		F_{(j,k)}(n) = 
		\begin{cases}
		0 &\quad\text{if}\quad 1\leq n\leq k-2\\
		1 &\quad\text{if}\quad n=k-1\\
		0 &\quad\text{if}\quad n=k\\
		F_{(j,k)}(n-j)+F_{(j,k)}(n-k) &\quad\text{if}\quad n\ge k+1\\
		\end{cases}
		\]
	\end{center}

	In $M_{C_n(d_1,d_2)}^*$, we have
	\begin{align*}
	[v_0] 	&= [v_j]+[v_k] \\
	&= [v_{2j}]+[v_k]+[v_{j+k}] \\
	&= [v_{3j}]+[v_k]+[v_{j+k}]+[v_{2j+k}] \\
	&= \dots
	\end{align*}
	The coefficients appearing in the above equations are terms in the sequence $F_{(d_1,d_2)}$ and corresponding $T_{C_n(d_1,d_2)}$ is given by the following 
	
	\begin{lemma}\label{lemma: T^n in terms of linear recurrence relation when S={i,j}}
		For fixed $d_1,d_2$, let $d_2-d_1=k$. Let $T=T_{C_n(d_1,d_2)}$. Suppose $G(n):=F_{(d_1,d_2)}(n)$ be the sequence defined above. Then for each $n\in\mathbb{N}$,
		\begin{equation*}
		T^{n}=\left(\begin{array}{cccc}
		
		G(n-1) & G(n) & \dots &  G(n+d_2-2) \\
		G(n-2) & G(n-1) & \dots &  G(n+d_2-3) \\
		\vdots & \vdots & \dots &  \vdots  \\
		\rowcolor{blue!20} 
		G(n+d_1-1) & G(n+d_1) & \dots &  G(n+d_2+d_1-2) \\
		\vdots & \vdots & \dots &  \vdots  \\
		
		G(n) & G(n+1) & \dots &  G(n+d_2-1) \\
		
		\end{array}\right)
		\end{equation*}
		where the highlighted row is $(k+1)^{th}$ row. 	
	\end{lemma}
	\begin{proof}
		We prove the lemma by induction on $n$. We extend the definition of $G$ to the negative integers as well. 
		Then, 
		\begin{equation*}
		T=\left(\begin{array}{ccc>{\columncolor{blue!20}}cccc}
		
		0 & 0 & \dots & 0 & \dots & 0 & 1 \\
		1 & 0 & \dots & 0 & \dots & 0 & 0 \\
		\vdots & \vdots & & \vdots & & \vdots & \vdots  \\
		\rowcolor{blue!20} 
		0 & 0 & \dots & 1 & \dots &  0 & 1 \\
		\vdots & \vdots &  & \vdots & &\vdots &\vdots \\
		0 & 0 & \dots & 0 & \dots & 0 &  0 \\
		0 & 0 & \dots & 0 & \dots & 1 &  0 \\
		
		\end{array}\right)
		\end{equation*}
		
		\begin{equation*}
		=\left(\begin{array}{ccc>{\columncolor{blue!20}}cccc}
		
		G(0) & G(1) & \dots & G(k-1) & \dots & G(d_2-2) & G(d_2-1) \\
		G(-1) & G(0) & \dots & G(k-2) & \dots & G(d_2-3) & G(d_2-2) \\
		\vdots & \vdots & & \vdots & & \vdots & \vdots  \\
		\rowcolor{blue!20} 
		G(d_1) & G(d_1+1) & \dots & G(d_2-1) & \dots &  G(d_2+d_1-2) & G(d_2+d_1-1) \\
		\vdots & \vdots &  & \vdots & &\vdots &\vdots \\
		G(2) & G(3) & \dots & G(k+1) & \dots & G(d_2) &  G(d_2+1) \\
		G(1) & G(2) & \dots & G(k) & \dots & G(d_2-1) &  G(d_2) \\
		
		\end{array}\right)
		\end{equation*}
		where highlighted column is $k^{th}$ column.

		Thus we have the statement true for $n=1$.
		Now suppose
		\begin{equation*}
		T^{n-1}=\left(\begin{array}{cccc}
		
		G(n-2) & G(n-1) & \dots &  G(n+d_2-3) \\
		G(n-3) & G(n-2) & \dots &  G(n+d_2-4) \\
		\vdots & \vdots & \dots &  \vdots  \\
		\rowcolor{blue!20} 
		G(n+d_1-2) & G(n+d_1-1) & \dots &  G(n+d_2+d_1-3) \\
		\vdots & \vdots & \dots &  \vdots  \\
		
		G(n-1) & G(n) & \dots &  G(n+d_2-2) \\
		
		\end{array}\right)
		\end{equation*}
		Then,
		\[T^{n}=T^{n-1}T\]
		\begin{equation*}
		=\left(\begin{array}{cc>{\columncolor{blue!20}}cccc}
		
		G(n-2) &  \dots & G(n+k-2) &\dots& G(n+d_2-3) \\
		G(n-3) &  \dots & G(n+k-3) &\dots& G(n+d_2-4) \\
		\vdots &  & \vdots & &\vdots \\
		\rowcolor{blue!20} 
		G(n+d_1-2) &  \dots & G(n+d_2-2) &\dots& G(n+d_2+d_1-3) \\
		\vdots &  & \vdots & &\vdots \\
		
		G(n-1) &  \dots & G(n+k-2) &\dots&  G(n+d_2-2) \\
		
		\end{array}\right)
		\left(\begin{array}{cccc}
		
		0 & 0 & \dots & 1 \\
		1 & 0 & \dots & 0 \\
		\vdots & \vdots & \dots & \vdots  \\
		\rowcolor{blue!20} 
		0 & 0 & \dots & 1 \\
		\vdots & \vdots & \dots & \vdots  \\
		0 & 0 & \dots &  0 \\
		
		\end{array}\right)
		\end{equation*}	
		\begin{equation*}
		=\left(\begin{array}{cccc}
		
		G(n-1) & G(n) & \dots &  G(n-2)+G(n+k-2) \\
		G(n-2) & G(n-1) & \dots & G(n-3)+G(n+k-3) \\
		\vdots & \vdots & \dots &  \vdots  \\
		\rowcolor{blue!20} 
		G(n+d_1-1) & G(n+d_1) & \dots &  G(n+d_1-2)+G(n+d_2-2) \\
		\vdots & \vdots & \dots &  \vdots  \\
		
		G(n) & G(n+1) & \dots &  G(n-1)+G(n+k-2) \\
		
		\end{array}\right)
		\end{equation*}
		\begin{equation*}
		=\left(\begin{array}{cccc}
		
		G(n-1) & G(n) & \dots &  G(n+d_2-2) \\
		G(n-2) & G(n-1) & \dots &  G(n+d_2-3) \\
		\vdots & \vdots & \dots &  \vdots  \\
		\rowcolor{blue!20} 
		G(n+d_1-1) & G(n+d_1) & \dots &  G(n+d_2+d_1-2) \\
		\vdots & \vdots & \dots &  \vdots  \\
		
		G(n) & G(n+1) & \dots &  G(n+d_2-1) \\
		
		\end{array}\right)
		\end{equation*}
		
	\end{proof}
	
	Using the determinant divisors theorem, the Smith normal form of $T_{C_n(d_1,d_2)}^n-I_k$ can be reduced to
	\[SNF(T_{C_n(d_1,d_2)}^n-I_k)=\begin{pmatrix}
	\alpha_1(n) & & & & \\
	& \frac{\alpha_2(n)}{\alpha_1(n)} & & & \\
	& & \ddots & & \\
	& & &\frac{\alpha_{d_2}(n)}{\alpha_{d_2-1}(n)} & \\
	\end{pmatrix}\] 
	where $\alpha_i$ is the greatest common divisor of the set of all $i\times i$ minors of $T_{C_n(d_1,d_2)}^n$.
	
\begin{example}\label{eg: n=6, S={2,3}}
	Let $n=6, d_1=2, d_2=3$. The corresponding Cayley graph is
		
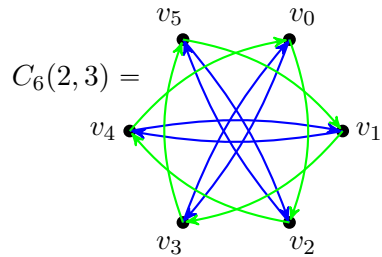
\begin{center}
			\begin{tikzpicture}
			[->,>=stealth',shorten >=1pt,thick,scale=0.7]

			\draw [fill=black] (0:2)		circle [radius=0.1];		
			\draw [fill=black] (60:2)		circle [radius=0.1];		
			\draw [fill=black] (120:2)		circle [radius=0.1];	
			\draw [fill=black] (180:2)		circle [radius=0.1];	
			\draw [fill=black] (240:2)		circle [radius=0.1];		
			\draw [fill=black] (300:2)		circle [radius=0.1];

			\path[->]	(0:2)	edge[blue, bend left=10] 	node	{}		(180:2);
			\path[->]	(60:2)	edge[blue, bend left=10] 	node	{}		(240:2);
			\path[->]	(120:2)	edge[blue, bend left=10] 	node	{}		(300:2);
			\path[->]	(180:2)	edge[blue, bend left=10] 	node	{}		(0:2);
			\path[->]	(240:2)	edge[blue, bend left=10] 	node	{}		(60:2);
			\path[->]	(300:2)	edge[blue, bend left=10] 	node	{}		(120:2);
			
			\path[->]	(60:2)	edge[green, bend left=20] 	node	{}		(300:2);
			\path[->]	(300:2)	edge[green, bend left=20] 	node	{}		(180:2);
			\path[->]	(180:2)	edge[green, bend left=20] 	node	{}		(60:2);
			\path[->]	(0:2)	edge[green, bend left=20] 	node	{}		(240:2);
			\path[->]	(240:2)	edge[green, bend left=20] 	node	{}		(120:2);
			\path[->]	(120:2)	edge[green, bend left=20] 	node	{}		(0:2);

			\node at (60:2.5) {$v_0$};
			\node at (0:2.5) {$v_1$};
			\node at (300:2.5) {$v_2$};
			\node at (240:2.5) {$v_3$};
			\node at (180:2.5) {$v_4$};
			\node at (120:2.5) {$v_5$};

			\node at (-3,1) {$C_6(2,3) =$};	
			
			\end{tikzpicture}
		\end{center}
		
		Corresponding companion matrix is given by
		\[T=\begin{pmatrix}
		0 & 0 & 1 \\
		1 & 0 & 1 \\
		0 & 1 & 0 \\
		
		\end{pmatrix}\]
		and,
		\[T^6-I_3=\begin{pmatrix}
		0 & 1 & 2 \\
		2 & 1 & 3 \\
		1 & 2 & 1 \\
		
		\end{pmatrix}\]
		whose Smith normal form is given by
		\[SNF(T^6-I_3)=\begin{pmatrix}
		1 & 0 & 0 \\
		0 & 1 & 0 \\
		0 & 0 & 7 \\
		
		\end{pmatrix}\]
		Hence,
		$K_0(L(C_6(2,3)))\cong\mathbb{Z}_7$ and $L(C_6(2,3))\cong L(1,8)$.

\end{example}

\section*{Acknowledgment}
	The author would like to thank B.Sury for fruitful discussions during the preparation of this paper. The author is grateful to Aditya Challa for his valuable help with  Python Programming Language. The author sincerely thanks Ramesh Sreekantan,  Roozbeh Hazrat, Gene Abrams, and Crist\'{o}bal Gil Canto for their very useful comments towards improving the paper. 
	
	The author gratefully acknowledges Department of Atomic Energy (National Board for Higher Mathematics), Government Of India for their financial support through Ph.D. Scholarship.

\end{document}